\documentclass[a4paper,11pt,UKenglish]{article}
\usepackage[top=3.2cm, bottom=3cm, left=3.45cm, right=3.45cm]{geometry}
\usepackage[T1]{fontenc}
\usepackage[utf8]{inputenc}
\usepackage{latexsym}
\usepackage{amsmath,amsthm,amssymb,amsfonts,mathtools}
\usepackage[]{babel}
\usepackage{tikz}
\usetikzlibrary{patterns, matrix}
\tikzstyle{NE-lines}=[pattern=north east lines, pattern color=black!45]
\usepackage{multirow}
\usepackage{booktabs}
\usepackage{multicol}
\usepackage{nicematrix}
\usepackage[shortlabels]{enumitem}
\usepackage{scalerel}
\usepackage[noblocks]{authblk}
\usepackage{placeins}

\DeclareMathOperator{\Img}{Im}

\newcommand{\quand}{\;\;\,\text{and}\,\;\;}
\newcommand{\alphas}{\alpha_n(S)}         
\newcommand{\betasw}{\kappa_n(S)}         
\newcommand{\betass}{\kappa^{\circ}_n(S)} 
\newcommand{\fubini}{\mathrm{fub}}
\newcommand{\fix}{\mathrm{Fix}}

\newcommand{\twocayhat}{\hat{B}}
\newcommand{\eps}{\epsilon}
\newcommand{\BB}{\mathbb{B}}
\newcommand{\LL}{\mathbb{L}}
\newcommand{\ZZ}{\mathbb{Z}}
\newcommand{\Cay}{\mathrm{Cay}}           
\newcommand{\Sym}{\mathcal{S}}           
\newcommand{\WI}{\mathrm{I}}              
\newcommand{\Asc}{A}                      
\newcommand{\Des}{D}                      
\newcommand{\asc}{\mathrm{asc}}           
\newcommand{\des}{\mathrm{des}}           
\newcommand{\id}{\mathrm{id}}

\newcommand{\ones}{\mathbf{1}}
\newcommand{\hproj}[1]{\hat{#1}\kern-1pt\cdot\kern-1pt\ones} 
\newcommand{\blocks}{\mathrm{blocks}}
\newcommand{\rows}{\mathrm{row}}
\newcommand{\cols}{\mathrm{col}}

\newcommand{\Bur}{\mathrm{Bur}}           
\newcommand{\BBur}{\Bur^{01}}             

\newcommand{\Mat}{\mathrm{Mat}}           
\newcommand{\BMat}{\Mat^{01}}             
\newcommand{\Genmat}{\mathfrak{Mat}}      
\newcommand{\BGenmat}{\Genmat^{01}}

\newcommand{\G}{\mathcal{G}}
\newcommand{\Par}{\mathrm{Par}}
\newcommand{\Bal}{\mathrm{Bal}}
\newcommand{\Cyc}{\mathcal{C}}
\newcommand{\Lin}{\mathit{L}}
\newcommand{\Set}{\mathit{E}}
\newcommand{\LCirc}{\Lin\circ(\Lin^m)_+}
\newcommand{\LCircStr}{\left(\LCirc\right)}
\DeclareMathOperator{\ModProd}{\setlength{\fboxsep}{.25\fboxsep}\boxed{\ast}}
\newcommand{\sgn}{\xi}

\newcommand{\textmultiset}[2]{\bigl(\kern-0.15em{\binom{#1}{#2}}\kern-0.15em\bigr)}
\newcommand{\displaymultiset}[2]{\left(\kern-0.35em{\binom{#1}{#2}}\kern-0.35em\right)}
\newcommand\multiset[2]{\mathchoice{\displaymultiset{#1}{#2}}
  {\textmultiset{#1}{#2}}
  {\textmultiset{#1}{#2}}
  {\textmultiset{#1}{#2}}}
\newcommand{\stone}[2]{\genfrac{[}{]}{0pt}{}{#1}{#2}}
\newcommand{\sttwo}[2]{\genfrac{\lbrace}{\rbrace}{0pt}{}{#1}{#2}}

\newtheorem{theorem}{Theorem}[section]
\newtheorem{proposition}[theorem]{Proposition}
\newtheorem{lemma}[theorem]{Lemma}

\newtheorem*{openproblem*}{Open Problem}

\theoremstyle{definition}
\newtheorem{definition}[theorem]{Definition}

\newtheorem*{remark*}{Remark}
\newtheorem{example}[theorem]{Example}
\newtheorem*{example*}{Example}

\title{Enumerative aspects of Caylerian polynomials}
\date{30 July 2025}
\author[1]{Giulio Cerbai\thanks{G.C. is a member of the Gruppo Nazionale Calcolo Scientifico--Istituto Nazionale di Alta Matematica (GNCS-INdAM).}}
\author[1]{Anders Claesson}

\affil[1]{Department of Mathematics, University of Iceland,
Reykjavik, Iceland, \texttt{akc@hi.is}, \texttt{giulio@hi.is}.}

\begin{document}

\thispagestyle{empty}
\maketitle

\begin{abstract}
  Eulerian polynomials record the distribution of descents over
  permutations. Caylerian polynomials likewise record the distribution
  of descents over Cayley permutations, where
  a Cayley permutation is a word of positive integers such
  that if a number appears in the word then all positive integers less than
  that number also appear in the word. Using combinatorial species and
  sign-reversing involutions we derive counting formulas and generating
  functions for the Caylerian polynomials as well as for related refined
  polynomials.
\end{abstract}

\section{Introduction}

\thispagestyle{empty}

A \emph{Cayley permutation} is a word of positive integers such
that if a number appears in the word then all positive integers less than
that number also appear in the word. We~\cite{CC2} recently introduced the \emph{Caylerian polynomials}, which
record the distribution of descents over Cayley permutations, and
provided a framework in which Burge matrices and Burge words underpin
these polynomials, allowing for generalizations of several results and
combinatorial constructions originally defined for the Eulerian
polynomials. In this sequel to the aforementioned paper~\cite{CC2}, we
apply our framework to obtain counting formulas and generating functions
for the Caylerian polynomials, as well as for some related polynomials
and combinatorial sets.

Our main results are as follows. In Theorem~\ref{thm_caylerian_explicit}
we provide an explicit formula for the $n$th Caylerian polynomial,
$C_n(t)$, in terms of Fubini numbers and Stirling numbers of the first
and second kind:
\[
  C_n(t) =
  \frac{1}{n!}\sum_{k,i}\fubini(k)\stone{n}{k}\sttwo{k}{i}i!(t-1)^{n-i}.
\]
Carlitz's~\cite{C2} identity for the Eulerian polynomials is
$tA_n(t)/(1-t)^{n+1} = \sum_{m \geq 1}m^nt^m$ and in
Theorem~\ref{theorem_genfun_n} we give an analogous identity for the
Caylerian polynomials:
\[
  \sum_{n\ge 0}\frac{tC_n(t)}{(1-t)^{n+1}}x^n =
  \sum_{m\ge 1}\frac{(1-x)^{m}}{2(1-x)^{m}-1}t^m.
\]
There is a natural two-sided generalization, $\twocayhat_n(s,t)$, of the
Caylerian polynomials and in Theorem~\ref{thm_twocay} we give a formula
for those polynomials:
\[
  \twocayhat_n(s,t) =
  \frac{1}{n!}\sum_{k=0}^n\stone{n}{k}\bigl|\Bal^s[k]\bigr|\bigl|\Bal^t[k]\bigr|
\]
Here
\[
  |\Bal^t[n]|=\sum_{w\in\Bal[n]}t^{\blocks(w)}
\]
is a polynomial recording the distribution of the number of blocks over
ballots (ordered set partitions); exact definitions will be given below.

In Section~\ref{sec_caylerian_polys} we recall the definitions of
(weak and strict) Caylerian polynomials and associated combinatorial
sets.  To simultaneously handle the structural and enumerative aspects
of these objects we shall use combinatorial species, which we briefly
introduce in Section~\ref{sec:L-species}. A particular species of
matrices of linear orders is defined in Section~\ref{section_burmat} and
through species equations we derive counting formulas and generating
functions.  In Sections~\ref{section_SRI_1}
and~\ref{section_SRI_2} we provide two sign-reversing involutions to
enumerate Cayley permutations with prescribed ascent set and binary
Burge matrices, respectively. In particular, we derive the following
formula (see Proposition~\ref{formulas_fixascset}) for the number of Cayley
permutations of $[n]$ whose ascent set is a subset of
$S=\{s_1,\dots,s_r\}$:
\[
  \betasw =
  \sum_{k,i}(-1)^i\binom{k}{i}\prod_{j=0}^r\binom{k-i}{s_{j+1}-s_{j}}.
\]
In Section~\ref{sec_final} we
make some closing remarks and summarize the formulas proved in this
paper; see tables~\ref{table_burge}, \ref{table_caylerian}
and~\ref{table_species}.

\section{Caylerian polynomials}\label{sec_caylerian_polys}

A \emph{Cayley permutation}
of $[n]$ is a map $w:[n]\to [n]$ with $\Img(w)=[k]$ for some
$k\leq n$, where $[0]=\emptyset$ and $[n]=\{1,2,\dots,n\}$ if $n\ge 1$,
and we identify the function $w$ with the word $w=w(1)\ldots w(n)$ whenever it is convenient.
Denote by $\Cay[n]$ the set of Cayley permutations on $[n]$.
For instance,
$\Cay[1]=\{ 1\}$, $\Cay[2]=\{ 11,12,21 \}$ and
$$
\Cay[3]=\{111,112,121,122,123,132,211,212,213,221,231,312,321\}.
$$
The term Cayley permutation was first used by Mor and Fraenkel~\cite{MF84} in
1983 and as the name suggests their history traces back to
Cayley. See the recent paper by Cerbai, Claesson, Ernst,
and Golab~\cite{CCEG}, for a short account of the history and the plethora
of guises that Cayley permutations have appeared under.

A \emph{ballot}, or \emph{ordered set partition}, of $[n]$ is a list of
disjoint blocks (nonempty sets) $B_1B_2\ldots B_k$ whose union is $[n]$.
Let $\Bal[n]$ be the set of ballots of $[n]$. It is well known that
Cayley permutations encode ballots: $B_1B_2\ldots B_k\in\Bal[n]$ is
encoded by $w\in\Cay[n]$
where $i\in B_{w(i)}$. For instance, $\{ 2,3,5\}\{ 6\}\{ 1,7\}\{ 4\}$
in $\Bal[7]$ is encoded by $3114123$ in $\Cay[7]$.
As a consequence, $|\Cay[n]|$ is equal to the $n$th
Fubini number; see sequence A000670 in the OEIS~\cite{Sl}.

The \emph{descent set} and
\emph{strict descent set} of $w\in\Cay[n]$ are defined, respectively, as
\begin{align*}
\Des(w) &= \{ i\in[n-1]: w(i)\ge w(i+1)\}; \\
\Des^{\circ}(w) &=\{ i\in[n-1]: w(i)>w(i+1)\}.
\end{align*}
We also let $\des(w)=|\Des(w)|$ and $\des^{\circ}(w)=|\Des^{\circ}(w)|$
denote their cardinalities. The sets $\Asc(w)$ and $\Asc^{\circ}(w)$
of \emph{ascents} and \emph{strict ascents}, as well as
$\asc(w)$ and $\asc^{\circ}(w)$, are defined analogously.
We will often add the word \emph{weak} to ascents and descents
to distinguish them from their strict counterparts.
The \emph{$n$th (weak)  Caylerian polynomial} and the \emph{$n$th
strict Caylerian polynomial} are defined by
$$
C_n(t)=\sum_{w\in\Cay[n]}t^{\des(w)}
\quand
C^{\circ}_n(t)=\sum_{w\in\Cay[n]}t^{\des^{\circ}(w)}.
$$
Let $w\in\Cay[n]$. Define the \emph{reverse}
of $w$ by $w^r(i)=w(n+1-i)$, for each $i\in[n]$. Also, define the \emph{complement} of $w$ by
$w^c(i)=\max(w)+1-w(i)$, for each $i\in[n]$, where $\max(w)=\max\{w(i): i\in [n]\}$.
The symmetry group generated by reverse and complement acts
on Cayley permutations, and for this reason we can replace
(weak) descents with (weak) ascents in the definition of the
Caylerian polynomials.
Furthermore, since $\des^{\circ}(w)=n-1-\asc(w)$, the coefficients
of the strict Caylerian polynomials $C^{\circ}_n(t)$ are simply
the reverse of the coefficients of $C_n(t)$:
$$
C^{\circ}_n(t)=t^{n-1}C_n\bigl(1/t\bigr).
$$
The resulting triangle of coefficients is A366173~\cite{Sl}.

The set $\Bur[n]$ of \emph{Burge words}~\cite{AlUh,Bu,CC,CC2} of size~$n$ is defined as
$$
\Bur[n]=
\{ (u,v)\in\WI[n]\times\Cay[n]: \Des(u)\subseteq \Des(v)\},
$$
where $\WI[n]$ denotes the set of weakly increasing Cayley
permutations of size~$n$.
A \emph{Burge matrix}~\cite{CC,CC2} is a matrix with nonnegative integer entries
whose every row and column has at least one nonzero entry.
The size of a Burge matrix is the sum of its entries and we let
$\Mat[n]$ denote the set of Burge matrices of size~$n$.
A bijection between $\Bur[n]$ and $\Mat[n]$ is obtained by
mapping the Burge word $(u,v)$ to the matrix $A=(a_{ij})$ where
$a_{ij}$ is equal to the number of pairs $(u(\ell),v(\ell))=(i,j)$
in $(u,v)$. For instance, the Burge word below corresponds to the Burge matrix
to its right:
$$
\binom{1\;1\;1\;2\;3\;3\;3\;3\;3}{3\;3\;1\;3\;4\;2\;2\;2\;2}
\quad
\begin{bmatrix}
  1 & 0 & 2 & 0 \\
  0 & 0 & 1 & 0 \\
  0 & 4 & 0 & 1
\end{bmatrix}.
$$
The inverse map is obtained by associating each matrix $A=(a_{ij})$
in $\Mat[n]$ with the biword $\binom{u}{v}$ of size $n$ where
any column $\binom{i}{j}$ appears $a_{ij}$ times, and the columns
are sorted in ascending order with respect to the top entry, breaking
ties by sorting in descending order with respect to the bottom entry.

Under the bijection described above, the set $\BMat[n]$
of \emph{binary Burge matrices} corresponds to
$$
\BBur[n] =
\bigl\{(u,v)\in\Bur[n]:\Des(u)\subseteq\Des^{\circ}(v)\bigr\}.
$$
The sequence of cardinalities $|\Mat[n]|=|\Bur[n]|$ is recorded as
A120733~\cite{Sl}, while its binary counterpart $|\BMat[n]|=|\BBur[n]|$
gives A101370.

There is a significant interplay between Cayley permutations and Burge
structures. Indeed, we~\cite[Theorem 5.1]{CC2} showed that:
\begin{align*}
C_n(2)         &= \bigl|\Bur[n]\bigr|  = \bigl|\Mat[n]\bigr|;\\
C^{\circ}_n(2) &= \bigl|\BBur[n]\bigr| = \bigl|\BMat[n]\bigr|,
\end{align*}
a generalization of the well-known fact that the $n$th Eulerian
polynomial evaluated at~$2$ is equal to $|\Bal[n]|$,
the $n$th Fubini number.
Pushing the interplay between Caylerian polynomials and
Burge structures further, we~\cite{CC2} defined the \emph{weak} and
\emph{strict two-sided Caylerian polynomials} by
\begin{align*}
\twocayhat_n(s,t) &= \sum_{A\in\Mat[n]}s^{\rows(A)}t^{\cols(A)};\\
\twocayhat^{\circ}_n(s,t) &= \sum_{A\in\BMat[n]}s^{\rows(A)}t^{\cols(A)},
\end{align*}
where $\rows(A)$ and $\cols(A)$ denote the number of rows and columns
of $A$, respectively, and showed that~\cite[Corollary 6.5]{CC2}:
\begin{align*}
C_n(t)&=(t-1)^n\twocayhat_n\left(1,\frac{1}{t-1}\right);\\
C^{\circ}_n(t)&=(t-1)^n\twocayhat^{\circ}_n\left(1,\frac{1}{t-1}\right).
\end{align*}

\section{$\LL$-species}\label{sec:L-species}

To delve deeper into the enumerative
aspects of Caylerian polynomials and Burge matrices we will
use combinatorial species.
The main references for the theory
of combinatorial species are Joyal's seminal paper~\cite{Joyal1981} and
the book by Bergeron, Labelle
and Leroux~\cite{BLL}. Shorter introductions can be found in the
papers by Claesson~\cite{Cl} and---in the context of (pattern-avoiding)
Cayley permutations---Cerbai et al.~\cite{CCEG}.

The prototypical combinatorial species are the
\emph{$\BB$-species}, and they are endofunctors on the category
$\BB$ whose objects are finite sets and whose morphisms are bijections.
We will, however, be using so called $\LL$-species, where $\LL$ denotes
the category of finite totally ordered sets with order-preserving
bijections as morphisms.

An \emph{$\LL$-species} is a functor $F:\LL\to\BB$ or, unwinding the
definition of a functor, an $\LL$-species is a rule $F$ that associates
\begin{itemize}
\item to each finite totally ordered set $\ell$, a finite set $F[\ell]$;
\item to each order preserving bijection $\sigma: \ell_1 \to \ell_2$, a
  bijection $F[\sigma]: F[\ell_1] \to F[\ell_2]$ such that
  $F[\sigma \circ \tau] = F[\sigma] \circ F[\tau]$ for all order
  preserving bijections $\sigma: \ell_1 \rightarrow \ell_2$,
  $\tau: \ell_2 \rightarrow \ell_3$, and $F[\id_{\ell}]=\id_{F[\ell]}$.
\end{itemize}
We will need the following species whose careful definitions
can be found in the aforementioned references~\cite{BLL,CCEG,Cl,Joyal1981}:
\begin{multicols}{2}
\begin{enumerate}[(a)]
\item $1$: characteristic of empty set;
\item $X$: singletons;
\item $\Set$: sets;
\item $\Set_+$: nonempty sets;
\item $\Lin$: linear orders;
\item $\Sym$: permutations;
\item $\Par$: set partitions;
\item $\Bal$: ballots.
\end{enumerate}
\end{multicols}
An element $s\in F[\ell]$ is called an \emph{$F$-structure} on $\ell$,
and the function $F[\sigma]$ is called the
\emph{transport of $F$-structures along $\sigma$}.
We shall often use $\ell=[n]$ with the natural order as our underlying totally
ordered set, and for ease of notation we write $F[n]=F[[n]]$.
The \emph{(exponential) generating function} of the species
$F$ is the formal power series
\[
   F(x) = \sum_{n \geq 0} |F[n]| \frac{x^n}{n!}.
\]
For instance, $1(x)=1$, $X(x)=x$, $\Set(x)=e^x$, and
$\Set_+(x)=e^x-1$. In general, if $F$ is a species, then $F_+$ denotes
the species of nonempty $F$-structures, and hence $F=1+F_+$.

Two species are considered (combinatorially) equal if there is a natural
isomorphism between them. Clearly, if $F=G$, then $F(x)=G(x)$. For
$\BB$-species, the converse is false, the typical example being
$\Lin\neq \Sym$, but $\Lin(x)=\Sym(x)=1/(1-x)$. For $\LL$-species it is,
however, the case that $F=G \Leftrightarrow F(x)=G(x)$. This is a
consequence of there being a unique order preserving bijection between
any given pair of finite totally ordered sets of the same
cardinality. So while transport of structure plays a crucial role in the
theory of $\BB$-species it is less important in the context of
$\LL$-species, and will usually be left implicit.
Henceforth, all species are assumed to be $\LL$-species and
we drop the prefix $\LL$.

We can construct new species from existing species using operations such
as addition, multiplication, cartesian multiplication, and composition.

The \emph{sum} of two species $F$ and $G$ is simple to define: for any finite
totally ordered set $\ell$, let
$(F+G)[\ell] = F[\ell] \sqcup G[\ell]$, where $\sqcup$ denotes disjoint
union. Clearly, $(F+G)(x)=F(x)+G(x)$.

For the \emph{product}, an
$(F\cdot G)$-structure on $\ell$ is a pair $(s,t)$, where $s$ is an
$F$-structure on a subset $\ell_1$ of $\ell$ and $t$ is a $G$-structure
on the remaining elements $\ell_2=\ell\setminus \ell_1$. That is,
$(F\cdot G)[\ell] = \bigsqcup (F[\ell_1] \times G[\ell_2])$ in which the
union is over all pairs $(\ell_1,\ell_2)$ such that
$\ell = \ell_1\cup \ell_2$ and $\ell_1\cap \ell_2 = \emptyset$.  It is
easy to see that $(FG)(x)=F(x)G(x)$.

The \emph{cartesian product} is defined by
$(F\times G)[\ell]=F[\ell]\times G[\ell]$. In other words, an
$(F\times G)$-structure on $\ell$ is the superposition of an
$F$-structure and a $G$-structure on $\ell$. Clearly,
$|(F\times G)[n]|=|F[n]|\!\cdot\!|G[n]|$ and in terms of exponential generating
functions this corresponds to the (exponential) Hadamard product of $F(x)$ and $G(x)$.

Assume $G[\emptyset]=\emptyset$, that is, there are no $G$-structures on
the empty set.  The \emph{(partitional) composition} of $F$ and $G$,
denoted $F\circ G$ or $F(G)$, is defined by
\[
(F\circ G)[\ell]\,=
\bigcup_{\beta\in\Par[\ell]}F[\beta]\times\prod_{B\in\beta}G[B],
\]
in which $\prod$ denotes the set-theoretical cartesian product.
In terms of cardinalities,
\begin{equation}\label{comp_formula}
|(F\circ G)[n]| \,= \sum_{\beta\in\Par[n]}|F[\beta]|\prod_{B\in\beta}|G[B]|
\end{equation}
and one can show that $(F \circ G)(x) = F(G(x))$.
Intuitively, an $(F\circ G)$-structure is a generalized
partition in which each block of the partition carries a $G$-structure,
and the blocks are structured by $F$.
For instance, a set partition is a set of nonempty sets, while a
ballot (ordered set partition) is a linear order of nonempty
sets:
$$
\Par = \Set\circ\Set_+
\quand
\Bal = \Lin\circ\Set_+,
$$
from which
$\Par(x)=\exp(e^x-1)$ and $\Bal(x)=1/(2-e^x)$
follow.

Cayley permutations can also be defined as a species~\cite{CCEG}:
\[
\Cay[\ell]= \bigl\{w\in[n]^\ell \!:\, \Img(w)=[k]\text{ for some } k\leq n\bigr\},\\
\]
where $n=|\ell|$ (and with transport of structure defined by
$\Cay[\sigma](w)=w \circ \sigma^{-1}$). As previously mentioned there is
a natural bijection between Cayley permutations and ballots, and hence
$\Cay=\Bal$.



Let us now return to the species $\Lin$ of linear orders. An
$\Lin$-structure on $\ell$ is a bijection $w:[n]\to\ell$, which we may
represent using one-line notation as $w=w(1)\ldots w(n)$. More generally,
an $\Lin^m$-structure on $\ell$ is an $m$-tuple, or vector, of linear
orders on disjoint underlying sets $\ell_1\cup\dots\cup\ell_m=\ell$.
And, an $(\Lin^m)_+$-structure is such a vector with at least one nonempty entry.
The generating function of $\Lin^m$ is
\begin{equation*}
\Lin^m(x)
=\frac{1}{(1-x)^m}
=(1+x+x^2+\cdots)^m
=\sum_{n\ge 0}\multiset{m}{n}x^n,
\end{equation*}
where the multichoose coefficient $\multiset{m}{n}=\binom{m+n-1}{n}$
is the number of multisets of cardinality $n$ over $[m]$. Thus,
\begin{equation}\label{eq_multiset}
  |\Lin^m[n]| = n!\multiset{m}{n}.
\end{equation}

\section{Matrices of linear orders}\label{section_burmat}

An $(\Lin\circ(\Lin^m)_+)$-structure is obtained by placing an
$(\Lin^m)_+$-structure on every block of a set partition of $\ell$, and
then a linear order on the blocks. Viewing an
$(\Lin^m)_+$-structure as a column vector with $m$ components, it
is natural to view an $\LCircStr$-structure as a matrix with $m$ rows.
We shall give an alternative description of those matrices, but first
we make a couple of definitions. For $(\alpha_1,\alpha_2)\in\Lin^2[\ell]$
we use juxtaposition, $\alpha_1\alpha_2\in\Lin[\ell]$, to denote the
concatenation of $\alpha_1$ and $\alpha_2$. For instance, if
$(\alpha_1,\alpha_2)=(451,32)\in\Lin^2[5]$, then
$\alpha_1\alpha_2=45132\in\Lin[5]$. Similarly, for a vector
$\mathbf{a}=(\alpha_1,\dots,\alpha_m)\in\Lin^m[\ell]$ of linear orders,
let $\prod\mathbf{a}=\alpha_1\ldots\alpha_m$.
We now have the following characterization of the matrices in question.

\begin{lemma}\label{lemma:matrix-characterization}
  Let
  $A=\bigl[\,\mathbf{a}_1\;\mathbf{a}_2\;\cdots\;\mathbf{a}_{k}\,\bigr]$
  be a matrix whose $i$th column is the vector $\mathbf{a}_i$ and whose
  entries are linear orders of disjoint sets. Then, $A$ is an
  $\LCircStr$-structure on $\ell$ if and only if $A$ has $m$ rows and
  it satisfies the following two conditions:
  \begin{enumerate}
  \item $\prod A := \prod\mathbf{a}_1\prod\mathbf{a}_2\,\cdots\,\prod\mathbf{a}_{k}\in\Lin[\ell]\hspace{0.3pt};$
  \item $\prod \mathbf{a}_j \neq \eps$\, for each $j\in [k]$.
  \end{enumerate}
\end{lemma}
In other words, a matrix $A$ with $m$ rows whose entries are linear
orders belongs to $\LCircStr[\ell]$ if different entries (linear orders)
have disjoint underlying sets, the union of the underlying sets is
$\ell$, and each column contains at least one nonempty linear order. For
instance, there are fourteen $(\Lin\circ(\Lin^2)_+)$-structures on $\{1,2\}$:
\begin{gather*}
  \begin{bmatrix}
    1 \\
    2
  \end{bmatrix}\;\,
  \begin{bmatrix}
    2 \\
    1
  \end{bmatrix}\;\,
  \begin{bmatrix}
    12 \\
    \eps
  \end{bmatrix}\;\,
  \begin{bmatrix}
    21 \\
    \eps
  \end{bmatrix}\;\,
  \begin{bmatrix}
    \eps \\
    12
  \end{bmatrix}\;\,
  \begin{bmatrix}
    \eps \\
    21
  \end{bmatrix} \\[1ex]
  \begin{bmatrix}
    1 & \eps \\
    \eps & 2
  \end{bmatrix}\;
  \begin{bmatrix}
    2 & \eps \\
    \eps & 1
  \end{bmatrix}\;
  \begin{bmatrix}
    \eps & 1 \\
    2 & \eps
  \end{bmatrix}\;
  \begin{bmatrix}
    \eps & 2 \\
    1 & \eps
  \end{bmatrix}\;
  \begin{bmatrix}
    1 & 2 \\
    \eps & \eps
  \end{bmatrix}\;
  \begin{bmatrix}
    2 & 1 \\
    \eps & \eps
  \end{bmatrix}\;
  \begin{bmatrix}
    \eps & \eps \\
    1 & 2
  \end{bmatrix}\;
  \begin{bmatrix}
    \eps & \eps \\
    2 & 1
  \end{bmatrix}.
\end{gather*}
Of these fourteen matrices seven are such that $\prod A=12$. In general,
the species of matrices in $\LCircStr[\ell]$ such that $\prod A$ is the
increasing linear order of $\ell$ can be defined as follows.

\begin{definition}\label{def:Genmat}
  For a totally ordered set $\ell=\{u_1,\dots,u_n\}$ with
  $u_1<\dots<u_n$ we define that an $m$-row matrix
  $A=\bigl[\,\mathbf{a}_1\;\mathbf{a}_2\;\cdots\;\mathbf{a}_{k}\,\bigr]$
  is an $\Genmat_{m}$-structure on~$\ell$ if its entries are linear
  orders of disjoint sets and it satisfies the following two conditions:
  \begin{enumerate}
  \item $\prod A = u_1u_2\ldots u_n$;
  \item $\prod \mathbf{a}_j \neq \eps$\, for each $j\in [k]$.
  \end{enumerate}
\end{definition}

In the prequel~\cite{CC2} the same notation, $\Genmat_{m}[n]$, was used
to denote the set of $m$-row matrices with nonnegative integer entries
whose total sum is $n$ and have at least one positive entry in each
column. The difference is mostly superficial and such matrices are
obtained from the matrices defined here by taking the length of each
entry. The slightly more complex definition given here better fits the
species framework we would like to work within.


We will now define a natural action of permutations on $\Genmat_m$
that will directly lead to our next simple lemma. Recall that a
permutation $w\in\Sym[\ell]$ is a bijection $w:\ell\to\ell$.
Permutations of a totally ordered set can also be represented in one-line notation:
$w=w(u_1)\ldots w(u_n)$, where $\ell=\{u_1,\dots,u_n\}$ and
$u_1<\dots<u_n$. Now, a permutation $w\in\Sym[\ell]$ acts on a
matrix $A=(a_{ij})\in\Genmat_m[\ell]$ by replacing each entry
$a_{ij}=c_1c_2\ldots c_r$ with $w\cdot a_{ij}:=w(c_1)w(c_2)\ldots w(c_r)$,
and we denote the resulting matrix by $w\cdot A=(w\cdot a_{ij})$.

\begin{lemma}\label{lemma_LxBur}
  For each $m\ge 0$,
  \[
    \Sym\times\Genmat_m = \LCirc.
  \]
\end{lemma}
\begin{proof}
  To prove this species identity, it suffices to give a size-preserving
  bijection from $\left(\Sym\times\Genmat_m\right)$-structures to
  $\LCircStr$-structures, and we claim that
  \[
    (w,A)\,\mapsto\, w\cdot A=:M
  \]
  is such a map. Indeed, it is easy to see that since $A$ satisfies the
  two conditions of Definition~\ref{def:Genmat}, the matrix $M$
  satisfies the corresponding two conditions of
  Lemma~\ref{lemma:matrix-characterization}.  Also, the inverse map is
  $M\mapsto (w,\, w^{-1}\cdot M)$ where $w=\prod M$.
\end{proof}

\begin{example}\label{example_wA}
Let $(w,A)\in\left(\Sym\times\Genmat_4\right)[9]$ be given by
$$
w=784652391
\quad\text{and}\quad
A=\begin{bmatrix}
\cdot & 5 & 67\\
123 & \cdot & \cdot\\
\cdot & \cdot & \cdot\\
4 & \cdot & 89
\end{bmatrix}.
$$
The bijection described in the proof of Lemma~\ref{lemma_LxBur}
maps $(w,A)$ to the matrix
$$
M=
\begin{bmatrix}
\cdot & 5 & 23\\
784 & \cdot & \cdot\\
\cdot & \cdot & \cdot\\
6 & \cdot & 91
\end{bmatrix}\in\bigl(\Lin\circ(\Lin^4)_+\bigr)[9].
$$
Note also that $w=\prod M$, $w^{-1}=967354128$, and $w^{-1}\cdot M = A$.
\end{example}

Next we shall derive a counting formula for $\Genmat_{m}$ from the
species equation given in Lemma~\ref{lemma_LxBur}, but first we need to
introduce some auxiliary notation.  Let us write $(a_1,\dots,a_k)\models n$ to
indicate that $(a_1,\dots,a_k)$ is a \emph{composition} of~$n$; that is,
each $a_i$ is a positive integer and $\sum_{i=1}^k a_i=n$. Moreover, for
$\alpha=(a_1,\dots,a_k)\models n$, let
\[
  \binom{n}{\alpha}=\binom{n}{a_1,\dots,a_k}=n!\,\big/\prod_{a\in\alpha}a!
\]
denote the associated multinomial coefficient. Now, for each $m,n\ge 0$,
\begin{align}
  |\Genmat_{m}[n]|
  &=\frac{1}{n!}\bigl|\LCircStr[n]\bigr| \nonumber\\[1ex]
  &=\frac{1}{n!}\sum_{\beta\in\Par[n]}|\Lin[\beta]|\prod_{B\in\beta}|\Lin^m[B]| \nonumber\\
  &=\frac{1}{n!}\sum_{\beta\in\Par[n]}k!\prod_{B\in\beta}\multiset{m}{|B|}|B|! \nonumber\\[0.9ex]
  &=\frac{1}{n!}\sum_{\alpha\models n}\binom{n}{\alpha}
    \prod_{a\in\alpha}\multiset{m}{a}a!
  =\sum_{\alpha\models n}\prod_{a\in\alpha}\multiset{m}{a}.\label{eq_genmat}
\end{align}
While this formula has a certain aesthetic appeal, it also has $2^{n-1}$
terms and is thus an ineffective way of counting these matrices. To
arrive at an effective formula we will need to derive another
species identity, but let us first give a simple general lemma
concerning the composition of species.

\begin{lemma}\label{lemma_species_comp}
  For any species $F$ and $G$, with $G[\emptyset]=\emptyset$,
  \[ |(F\circ G)[n]| \,=\,\sum_{k=0}^n\,|F[k]|\cdot |E_k(G)[n]|,
  \]
  where $E_k$ denotes the species characteristic of sets of cardinality
  $k$, defined by $E_k[\ell]=\{\ell\}$ if $|\ell|=k$ and
  $E_k[\ell]=\emptyset$ otherwise.
\end{lemma}
\begin{proof}
  We have
  \begin{align*}
    |(F\circ G)[n]|
    &= \sum_{\beta\in\Par[n]} |F[\beta]|\prod_{B\in\beta}|G[B]| \\
    &= \sum_{k=0}^n|F[k]|\!\sum_{\beta\in\Par[n]}|E_k[\beta]|\prod_{B\in\beta}|G[B]|
    = \sum_{k=0}^n|F[k]|\cdot |E_k(G)[n]|.\qedhere\\
  \end{align*}
\end{proof}

In further preparation for the next result, let us introduce the species of connected
linear orders. Suppose that the species $F$ and $G$
are related by $F=\Set(G)$, so that an $F$-structure is a set of
$G$-structures. We may then call $G$ the species of
\emph{connected $F$-structures}, denoted by $G=F^c$.
A classic example is $\Sym=\Set(\Cyc)$, where $\Cyc$ is the species
of cycles. That is, a connected permutation is a cycle, and any permutation
decomposes uniquely as a set of cycles.
The $(n,k)$th Stirling number of the first kind counts
permutations of $[n]$ with $k$ cycles:
\begin{equation*}
\stone{n}{k}=|\Set_k(\Cyc)[n]|.
\end{equation*}
A similar decomposition is obtained by splitting linear orders
by their left-to-right minima. Here, a \emph{left-to-right
minimum} of a linear order~$w$ is an entry $w(i)$ such that $w(i)<w(j)$ for
each $j<i$. For instance, the linear order $w=784652391$ of
Example~\ref{example_wA} decomposes as
$$
w\;=\;78\,|\,465\,|\,239\,|\,1.
$$
In this sense, a connected linear order is a linear order that begins
with the minimum of the underlying set, and any linear order decomposes
uniquely as a set of connected linear orders. To reconstruct the linear
order we juxtapose the connected components with their minima in
decreasing order, so that the minima of the connected components are the
left-to-right minima of the resulting linear order. In particular, from
$L(x)=1/(1-x)$, $E(x)=e^x$, and $\Lin=\Set(\Lin^c)$, it follows that
\[
  \Lin^c(x)=\log\left(\frac{1}{1-x}\right).
\]
\begin{proposition}\label{prop_bal_mLc}
  For each $n,m\ge 0$,
  $$
  \Sym\times\Genmat_m=\Bal(m\Lin^c)
  \quad\text{and}\quad
  |\Genmat_{m}[n]|=\frac{1}{n!}\sum_{k=0}^n \stone{n}{k}m^k \fubini(k),
  $$
  where $\fubini(k)$ denotes the $k$th Fubini number.
\end{proposition}
\begin{proof}
  By Lemma~\ref{lemma_LxBur} it suffices to prove that
  $\LCirc=\Bal(m\Lin^c)$.  We have
  \begin{align*}
    \LCirc
    &= \Lin\circ\bigl(\Lin^m-1\bigr)  \\
    &= \Lin\circ\bigl(\Set(\Lin^c)^m-1\bigr) \\ 
    &= \Lin\circ\bigl(\Set(m\Lin^c)-1\bigr) \\ 
    &=\Lin\circ\Set_+\circ (m\Lin^c) \\
    &=\Bal(m\Lin^c).
    \shortintertext{Furthermore,}
    n!\bigl|\Genmat_{m}[n]\bigr|
    &= \bigl|(\Bal(mX)\circ\Lin^c)[n]\bigr| \\
    &= \sum_{k=0}^n\,|\Bal(mX)[k]|\cdot |E_k(\Lin^c)[n]| \\ 
    &= \sum_{k=0}^n\,|(\Bal\times E(mX))[k]|\cdot |E_k(\Lin^c)[n]| \\
    &=\sum_{k=0}^n m^k\fubini(k) \stone{n}{k},
  \end{align*}
  where the second equality follows from Lemma~\ref{lemma_species_comp},
  and the third equality uses the general identity
  $F(mX)=F \times E(mX)$ for any species $F$.
\end{proof}

In the proof of Proposition~\ref{prop_bal_mLc} we derived the identity
$\LCirc=\Bal(m\Lin^c)$ by a sequence of species equations. We shall now
give an alternative bijective proof of this identity.  To an
$\LCircStr$-structure $M$ we associate a $\Bal(m\Lin^c)$-structure $B$
as follows.  Each entry of the matrix $M$ is a linear order and by
splitting each entry according to its left-to-right minima we obtain a
set of connected linear orders.  For instance, consider the matrix
$$
M=\begin{bmatrix}
\cdot & 5 & 23\\
784 & \cdot & \cdot\\
\cdot & \cdot & \cdot\\
6 & \cdot & 91
\end{bmatrix}\in\LCircStr[9]
$$
of Example~\ref{example_wA}. Here, $784$ splits into $78|4$ and $91$
splits into $9|1$, and the resulting set of connected linear orders is
$$
\lbrace 78,4,6,5,23,9,1 \rbrace.
$$
Let us call a connected linear order arising this way an \emph{atom}. To
determine $B$, we need to place a $(\Bal\times\Set(mX))$-structure on
the set of atoms. The ballot is defined by collecting in the $j$th block
the atoms contained in the $j$th column of $M$.  To place an
$\Set(mX)$-structure means to assign a color in $[m]$, and we let the
color of each atom be equal to its row index in $M$.  Here, we obtain
$$
B=
\lbrace 78_2,4_2,6_4\rbrace
\lbrace 5_1\rbrace
\lbrace 23_1,9_4,1_4\rbrace\in\Bal(m\Lin^c)[9],
$$
where colors are indicated by the subscripts.  To prove that the map
$M\mapsto B$ is a bijection, we describe its inverse. Given $B$, the
$(i,j)$th entry of $M$ is the linear order obtained by sorting the
$i$-colored atoms in the $j$th block of $B$ decreasingly with respect to
their first letter and then concatenating them.

By composing the map of Lemma~\ref{lemma_LxBur} with the map
described above, we obtained a bijection between $\Sym\times\Genmat_m$
and $\Bal(m\Lin^c)$, which leads to a formula for $|\Genmat_m[n]|$.
It is easy to adjust this approach to prove an analogous equation
for Burge matrices and obtain a formula for $|\Mat[n]|$.

\begin{proposition}\label{prop_spec_burmat}
For each $n\ge 0$,
$$
\Sym\times\Mat=
\left(\Bal\times\Bal\right)\circ\Lin^c
\quad\text{and}\quad
|\Mat[n]|=\frac{1}{n!}\sum_{k=0}^n\fubini(k)^2\stone{n}{k}.
$$
\end{proposition}
\begin{proof}
In Lemma~\ref{lemma_LxBur} and in the previous paragraph we
gave two bijections
$$
(w,A)\longleftrightarrow M
\quad\text{and}\quad
M\longleftrightarrow B,
$$
where $(w,A)$ is an $\left(\Sym\times\Genmat_m\right)$-structure,
$M$ is an $\LCircStr$-structure, and $B$ is a $\Bal(m\Lin^c)$-structure.
Recall that a matrix $A$ of $\Genmat_m$ is a Burge matrix if and
only if it has no empty rows, which in turn is equivalent to
each row of $M$ containing at least one nonempty linear
order.
In terms of $\Bal(m\Lin^c)$, the structure determining the row
indices of the atoms in each column is more restrictive than $\Set(mX)$,
since empty rows are not allowed. In fact, it is the same structure as
the one used for the (nonempty) columns, that is, the ballot defined
by collecting in the $i$th block the atoms contained in the $i$th
row of~$A$. Conversely, the $(i,j)$th entry of~$A$ contains the linear
order obtained by juxtaposing the atoms that are simultaneously in the
$i$th block of the ballot determining the row indices and in the $j$th
block of the ballot determining the column indices.
Consequently, a species identity for $\Sym\times\Mat$ is obtained
by simply replacing $\Set(mX)$ with $\Bal$:
\begin{align*}
\Sym\times\Mat &=
\left(\Bal\times\Bal\right)\circ\Lin^c.
\intertext{In particular,}
  n!|\Mat[n]|
  &= \bigl|\bigl(\bigl(\Bal\times\Bal\bigr)\circ\Lin^c\bigr)[n]\bigr| \\[0.8ex]
  &= \sum_{k=0}^n\,\bigl|\bigl(\Bal\times\Bal\bigr)[k]\bigr|\cdot \bigl|E_k(\Lin^c)[n]\bigr| && \text{(by Lemma~\ref{lemma_species_comp})}\\
  & =\sum_{k=0}^n\fubini(k)^2\stone{n}{k}.&&\qedhere
\end{align*}
\end{proof}

We wish to refine the previous result to obtain a formula
for the (weak) two-sided Caylerian polynomials
$$
\twocayhat_n(s,t)=\sum_{A\in\Mat[n]}s^{\rows(A)}t^{\cols(A)}
$$
defined in Section~\ref{sec_caylerian_polys}.
Let $\Mat^{s,t}$ be the ($\ZZ[s,t]$-weighted) species of Burge matrices where
every row is marked by~$s$ and every column is marked by~$t$. More
formally, we assign to a matrix $A\in\Mat[n]$ the weight
$s^{\rows(A)}t^{\cols(A)}$, so that
\begin{equation}\label{eq_weightmat}
\twocayhat_n(s,t)=\sum_{A\in\Mat[n]}s^{\rows(A)}t^{\cols(A)}
=|\Mat^{s,t}[n]|
\end{equation}
is the total weight of all the matrices in $\Mat[n]$.
From the proof of Proposition~\ref{prop_spec_burmat} it follows that the joint
distribution of rows and columns on Burge matrices can be obtained by
tracking the number of blocks in $(\Bal\times\Bal)$-structures. Let $\Bal^t$ be
the species of ballots where each block is marked by~$t$. Then
$$
|\Bal^t[n]|=\sum_{w\in\Bal[n]}t^{\blocks(w)}
\quad\text{and}\quad
\Bal^t(x)=\Lin(t\Set_+)(x)=\frac{1}{1-t(e^x-1)}.
$$
Expressed differently, $\Bal^t=L\circ tX \circ E_+$, where $tX$ denotes
the species of singletons where each singleton is marked with a $t$. In
particular, the equation
\begin{align}
\Sym\times\Mat &= \left(\Bal\times\Bal\right)\circ\Lin^c\nonumber
\shortintertext{is refined by}
\Sym\times\Mat^{s,t} &= \left(\Bal^s\times\Bal^t\right)\circ\Lin^c.\label{eq_symXmatst}
\end{align}
Using Equation~\eqref{eq_weightmat}, we immediately obtain the
following identity for the weak two-sided Caylerian polynomials:

\begin{theorem}\label{thm_twocay}
  We have
  \[
    \twocayhat_n(s,t) =
    \frac{1}{n!}\sum_{k=0}^n\stone{n}{k}\bigl|\Bal^s[k]\bigr|\bigl|\Bal^t[k]\bigr|.
  \]
\end{theorem}

This identity, in turn, leads to the explicit formula for $C_n(t)$:

\begin{theorem}\label{thm_caylerian_explicit}
  The $n$th Caylerian polynomial satisfies
  \[
    C_n(t) =
    \frac{1}{n!}\sum_{k=0}^n\stone{n}{k}\fubini(k)\sum_{i=0}^k\sttwo{k}{i}i!(t-1)^{n-i},
  \]
  where $\sttwo{n}{k}=|\Set_k(\Set_+)[n]|$ denotes the $(n,k)$th
  Stirling number of the second kind, counting set partitions of $[n]$
  with $k$ blocks.
\end{theorem}
\begin{proof}
  We have
  \begin{align*}
    C_n(t)
    &=(t-1)^n\twocayhat_n\biggl(1,\frac{1}{t-1}\biggr) && \text{\cite[Corollary 6.5]{CC2}} \\
    &=(t-1)^n\frac{1}{n!}\sum_{k=0}^n\stone{n}{k}\bigl|\Bal^1[k]\bigr|\bigl|\Bal^{\frac{1}{t-1}}[k]\bigr| \\
    &=(t-1)^n\frac{1}{n!}\sum_{k=0}^n\stone{n}{k}\fubini(k)
      \!\left(\sum_{i=0}^k\sttwo{k}{i}i!\frac{1}{(t-1)^i}\right) && \qedhere
  \end{align*}
\end{proof}


The results proven thus far can be extended to binary matrices,
as we sketch below.

\begin{proposition}\label{prop_enum_bgenm}
For each $n,m\ge 0$,
$$
\Sym\times\BGenmat_m=\Lin\bigl((1+X)^m-1\bigr)
\quad\text{and}\quad
|\BGenmat_{m}[n]|=\sum_{\alpha\models n}\prod_{a\in\alpha}\binom{m}{a}.
$$
\end{proposition}
\begin{proof}
By restricting the size-preserving bijection from
$\left(\Sym\times\Genmat_m\right)$-structures to
$\LCircStr$-structures described in Lemma~\ref{lemma_LxBur}
to pairs $(w,A)$ where $A$ is binary, one obtains matrices
$M=w\cdot A$ whose entries are linear orders of size at most one.
In other words, each column is a nonempty $(1+X)^m$-structure
and we get
$$
\Sym\times\BGenmat_m=\Lin\bigl((1+X)^m-1\bigr).
$$
The proof of the equation for $|\BGenmat_{m}[n]|$ is similar to
the proof of Equation~\eqref{eq_genmat} for $|\Genmat_m[n]|$.
The multichoose coefficient
$\multiset{m}{n}$
is, however, replaced by the binomial coefficient
$\binom{m}{n}$
since the generating function
$1/(1-x)^m=\sum_{n\ge 0}\multiset{m}{n}x^n$
is replaced by the generating function
$(1+x)^m=\sum_{n\ge 0}\binom{m}{n}x^n$.
\end{proof}

Next, we wish to compute the \emph{ordinary} generating functions of the
two sequences $|\Genmat_{m}[n]|$ and $|\BGenmat_{m}[n]|$. In general, if
$F$ is a species, then the ordinary generating function
$\sum_{n\geq 0}|F[n]|x^n$ is equal to the (exponential) generating
function of $\Sym\times F$:
\[\sum_{n\geq 0}|F[n]|x^n = \sum_{n\geq 0}n!|F[n]|\frac{x^n}{n!} = (\Sym\times F)(x).
\]
\begin{proposition}\label{prop:ogf_genmat}
  We have
  \[
    \sum_{n\ge 0}|\Genmat_{m}[n]|x^n  =\frac{(1-x)^m}{2(1-x)^m-1}\quad\text{and}\quad
    \sum_{n\ge 0}|\BGenmat_{m}[n]|x^n =\frac{1}{2-(1+x)^m}.
  \]
\end{proposition}
\begin{proof}
  Using Lemma~\ref{lemma_LxBur} and Proposition~\ref{prop_enum_bgenm},
  respectively, we get
  \begin{align*}
    \sum_{n\ge 0}|\Genmat_{m}[n]|x^n
    &=\bigl(\Sym\times\Genmat_m\bigr)(x)\\[-1.4ex]
    &=\LCircStr(x) \\[0.6ex]
    &=\frac{1}{1-\bigl((1-x)^{-m}-1\bigr)}
    =\frac{(1-x)^m}{2(1-x)^m-1}
  \shortintertext{and}
    \sum_{n\ge 0}|\BGenmat_{m}[n]|x^n
    &=\bigl(\Sym\times\BGenmat_m\bigr)(x) \\[-1.4ex]
    &=\Lin\bigl((1+x)^m-1\bigr) \\[0.6ex]
    &=\frac{1}{1-\bigl((1+x)^m-1\bigr)}
    =\frac{1}{2-(1+x)^m}.&&\qedhere
  \end{align*}
\end{proof}

To prove the next result, we will need two equations involving the
weak and strict Caylerian polynomials~\cite[Theorem 7.6]{CC2}:
\begin{align}
\frac{tC_n(t)}{(1-t)^{n+1}}&=\sum_{m\ge 1}|\Genmat_{m}[n]|t^m;
\label{eq_caypoly1}\\
\frac{tC^{\circ}_n(t)}{(1-t)^{n+1}}&=\sum_{m\ge 1}|\BGenmat_{m}[n]|t^m.
\label{eq_caypoly2}
\end{align}
These equations were originally stated in terms of (generalized)
Burge words. For convenience, we have here reformulated them in
terms of the equinumerous structures $\Genmat_m$ and $\BGenmat_m$.
We note that identities~\eqref{eq_caypoly1} and \eqref{eq_caypoly2}
have the same flavor as the following identity due to
Carlitz~\cite{C2}, often used as an alternative definition of the
Eulerian polynomials $A_n(t)$:
$$
\frac{tA_n(t)}{(1-t)^{n+1}} = \sum_{m \geq 1}m^nt^m.
$$

\begin{theorem}\label{theorem_genfun_n}
We have
\begin{align*}
  \sum_{n\ge 0}\frac{tC_n(t)}{(1-t)^{n+1}}x^n
  &= \sum_{m\ge 1}\frac{(1-x)^{m}}{2(1-x)^{m}-1}t^m
\shortintertext{and}
  \sum_{n\ge 0}\frac{tC^{\circ}_n(t)}{(1-t)^{n+1}}x^n
  &= \sum_{m\ge 1}\frac{1}{2-(1+x)^{m}}t^m.
\end{align*}
\end{theorem}
\begin{proof}
Using Proposition~\ref{prop:ogf_genmat} and Equation~\eqref{eq_caypoly1},
\begin{align*}
\sum_{n\ge 0}\frac{tC_n(t)}{(1-t)^{n+1}}x^n
&=\sum_{n\ge 0}\sum_{m\ge 1}|\Genmat_{m}[n]|t^{m}x^n\\
&=\sum_{m\ge 1}\left(\sum_{n\ge 0}|\Genmat_{m}[n]|x^n\right)t^{m}
=\sum_{m\ge 1}\frac{(1-x)^m}{2(1-x)^m-1}t^{m}.
\end{align*}
The second identity is similarly obtained by
Proposition~\ref{prop:ogf_genmat} and Equation~\eqref{eq_caypoly2}.
\end{proof}

\section{Cayley permutations with a prescribed ascent set}\label{section_SRI_1}

In Section~\ref{sec_caylerian_polys}, we defined the weak ascent set
$\Asc(w)$ and the strict ascent set $\Asc^{\circ}(w)$ of a Cayley
permutation~$w$, and we shall use the same definitions here in the
context of linear orders.
Let~$n$ be a positive integer. Consider a subset
$S=\lbrace s_1,\dots,s_r\rbrace$ of $[n-1]$, with
$s_1<s_2<\dots<s_r$, and let
$$
\alphas=|\lbrace w\in\Lin[n]:\Asc(v)\subseteq S\rbrace|
$$
be the number of linear orders on $[n]$ whose ascent set is contained
in~$S$. Then (see MacMahon~\cite[vol.~1, p.~190]{MacMahon} or Stanley~\cite[Proposition 1.4.1]{St}),
$\alphas$ is a polynomial in $n$ given by the multinomial coefficient
\[
\alphas
  = \binom{n}{s_1,s_2-s_1,s_3-s_2,\dots,n-s_r}
  = \frac{n!}{s_1!(s_2-s_1)!(s_3-s_2)!\cdots(n-s_r)!}.
\]
Indeed, to obtain any $w\in\Lin[n]$ with $\Asc(w)\subseteq S$, first pick
$s_1$ elements $w(1)>w(2)>\dots>w(s_1)$ from $[n]$, then $s_2-s_1$ elements
$w(s_1+1)>w(s_1+2)>\dots>w(s_2)$ from the remaining elements, and so on.
Equivalently, $\alphas$ counts ballots on $[n]$ whose block sizes are
given by the vector
$$
\Delta(S):=(s_1,s_2-s_1,\dots,n-s_r).
$$
We~\cite{CC2} have extended $\alphas$ to Cayley permutations by letting
\begin{align*}
\betasw &=|\lbrace w\in\Cay[n]:\, \Asc(w)\subseteq S\rbrace|;\\
\betass &=|\lbrace w\in\Cay[n]:\, \Asc^{\circ}(w)\subseteq S\rbrace|.
\end{align*}
We~\cite[Theorem 8.2]{CC2} then showed that
\begin{equation}\label{eq_mat(S)}
\betass=|\Mat(S)[n]|
\quand
\betasw=|\BMat(S)[n]|,
\end{equation}
where $\Mat(S)$ denotes the set of Burge matrices whose vector
of row sums is equal to $\Delta(S)$,
and $\BMat(S)$ denotes such matrices that are binary.

The main goal of this section is to establish the following result.

\begin{proposition}\label{formulas_fixascset}
We have
\begin{align}
|\Genmat_{m}[n]| &= 
\sum_{k=0}^n\sum_{i=0}^k(-1)^i\binom{k}{i}\multiset{m(k-i)}{n};\label{eq_sign_rev_Burge}\\
|\BGenmat_{m}[n]| &=
\sum_{k=0}^n\sum_{i=0}^k(-1)^i\binom{k}{i}\binom{m(k-i)}{n}.
\label{eq_sign_rev_BinBurge}\\
\intertext{Furthermore, for $S=\{s_1,\dots,s_r\}\subseteq [n-1]$ with $s_1<s_2<\dots<s_r$,}
\betass &=
\sum_{k=0}^n\sum_{i=0}^k
(-1)^i\binom{k}{i}\prod_{j=0}^r\multiset{k-i}{s_{j+1}-s_{j}};\label{eq_sign_rev_betass}\\
\betasw &=
\sum_{k=0}^n\sum_{i=0}^k
(-1)^i\binom{k}{i}\prod_{j=0}^r\binom{k-i}{s_{j+1}-s_{j}},\label{eq_sign_rev_betasw}
\end{align}
where $s_0=0$ and $s_{r+1}=n$.
\end{proposition}

Munarini, Poneti and Rinaldi~\cite{MPR} used the
inclusion–exclusion principle to obtain identities~\eqref{eq_sign_rev_Burge}
and~\eqref{eq_sign_rev_betass} in the context of composition
matrices (see equation~(12) and equation~(13) in their paper,
respectively). Equation~\eqref{eq_sign_rev_betass} also appears
in a book by Andrews~\cite[equation~(4.3.3)]{And} as a counting
formula for so-called \emph{vector compositions}, structures
previously studied by MacMahon~\cite{MacM} under the name of
\emph{compositions of multipartite numbers}.
Here, we provide sign-reversing involution proofs of these equations
that allow us to generalize them to binary matrices $\BGenmat_{m}$
and to $\betasw$. Let us start with the first two equations,
\eqref{eq_sign_rev_Burge} and \eqref{eq_sign_rev_BinBurge}.

For $m\geq 0$, define $\G_m$ as the species obtained from $\Genmat_{m}$
by allowing empty columns, but no more than $n$ columns in total, and
then giving each column a sign such that nonempty columns are positive,
but empty columns can be negative or positive. More precisely, for a
totally ordered set $\ell=\{u_1,\dots,u_n\}$ with $u_1<\dots<u_n$, we
define a $\G_{m}$-structure on~$\ell$ as a pair $(A,\sgn)$, where
$A=\bigl[\,\mathbf{a}_1\;\mathbf{a}_2\;\cdots\;\mathbf{a}_{k}\,\bigr]$
is a matrix with $m$ rows and $k\leq n$ columns, $\sgn:[k]\to\{-1,1\}$,
and the following two conditions are satisfied:
\begin{enumerate}
\item $\prod A = u_1u_2\ldots u_n$;
\item $\prod \mathbf{a}_i\neq\epsilon\implies \sgn(i)=1$.
\end{enumerate}
For simplicity, we will refer to the pair $(A,\sgn)$ simply as a
(signed) matrix and use $A$ as its identifier.  By way of illustration,
we may write $A\in\G_{4}[9]$ where the matrix $A$ together with the
values of $\sgn$ (on top of each column) are displayed below:
\[
  A=
  \begin{bNiceMatrix}[first-row,margin]
    -     & +     & +     & -     & -     & +     & + \\
    \cdot & \cdot & \cdot & \cdot & \cdot & 5     & 67 \\
    \cdot & 123   & \cdot & \cdot & \cdot & \cdot & \cdot \\
    \cdot & \cdot & \cdot & \cdot & \cdot & \cdot & \cdot \\
    \cdot & 4     & \cdot & \cdot & \cdot & \cdot & 89
  \end{bNiceMatrix}
\]
As an additional illustration, the five matrices in $\G_1[2]$ are displayed below:
\[
\begin{bNiceMatrix}[first-row,margin]
  + & + \\
  1 & 2
\end{bNiceMatrix}\qquad
\begin{bNiceMatrix}[first-row,margin]
  +     & +  \\
  \cdot & 12
\end{bNiceMatrix}\qquad
\begin{bNiceMatrix}[first-row,margin]
  -     & +  \\
  \cdot & 12
\end{bNiceMatrix}\qquad
\begin{bNiceMatrix}[first-row,margin]
  +  & +  \\
  12 & \cdot
\end{bNiceMatrix}\qquad
\begin{bNiceMatrix}[first-row,margin]
  +  & -  \\
  12 & \cdot
\end{bNiceMatrix}
\]
To construct a matrix $A$ in $\G_m[n]$ with $k$ columns, first choose
$i\leq k$ negative columns in $\binom{k}{i}$ ways. Next, fill in the
remaining $mk-mi=m(k-i)$ entries of the matrix with linear orders so
that $\prod A = 12\ldots n$. It is enough to pick the sizes of the
linear orders and they form a weak composition of $n$ into $m(k-i)$
parts; that is, an integer composition with nonnegative parts. There are
$\multiset{m(k-i)}{n}$ such compositions and we have
arrived at the formula
$$
|\G_m[n]|=\sum_{k=0}^n\sum_{i=0}^k\binom{k}{i}\multiset{m(k-i)}{n}.
$$

By abuse of notation, we define the sign of
$A=\bigl[\,\mathbf{a}_1\;\mathbf{a}_2\;\cdots\;\mathbf{a}_{k}\,\bigr]\in\G_m[n]$
by
$$
\sgn(A)=\sgn(1)\sgn(2)\cdots\sgn(k)
$$
In other words, $\sgn(A)=(-1)^i$ where $i$ is the number of negative
columns. Now, given $A\in \G_m[n]$, let $\lambda(A)\in \{0,1,\ldots,k\}$
be the index of the leftmost empty column of $A$; if there is no such
column, let $\lambda(A)=0$. Note that $A\in\Genmat_{m}[n]$ if and
only if $\lambda(A)=0$. Finally, define the map
$\gamma:\G_m[n]\to \G_m[n]$ by letting $\gamma(A)$ be the matrix
obtained from $A$ by reversing the sign of column $\lambda(A)$ if
$\lambda(A)\neq 0$, and letting $\gamma(A)=A$ otherwise. It is clear
that~$\gamma$ is an involution on $\G_m[n]$:
$$
\gamma\bigl(\gamma(A)\bigr)=A.
$$
Furthermore, let $\fix(\gamma)=\lbrace A:\gamma(A)=A\rbrace$ be the
set of \emph{fixed points} of $\gamma$. Then
$$
\mathrm{Fix}(\gamma)=\{A: \lambda(A)=0\}=\Genmat_{m}[n].
$$
On the other hand, if $A$ is not fixed, then
$\sgn(\gamma(A))=-\sgn(A)$. Consequently, $\gamma$ is a sign-reversing
involution on $\G_m[n]$ whose set of fixed points is
$\Genmat_{m}[n]$. Since fixed points have positive sign,
Equation~\eqref{eq_sign_rev_Burge} follows:
\begin{align*}
|\Genmat_{m}[n]|
&=\sum_{A\in\fix(\gamma)}\sgn(A)\\
&=\sum_{A\in \G_m[n]}\sgn(A)
=\sum_{k=0}^n\sum_{i=0}^k(-1)^i\binom{k}{i}\multiset{m(k-i)}{n}.
\end{align*}

The same technique applies to binary matrices in $\BGenmat_{m}[n]$.
Here, the nonempty linear orders have size one and hence the weak
composition in the above argument has parts of size at most one.
Clearly, there are $\binom{m(k-i)}{n}$ such compositions of $n$ into
$m(k-i)$ parts. Following what has by now become a familiar pattern in
this article, the formula for $\BGenmat_m$ is simply obtained by
replacing the multichoose coefficient with a binomial coefficient in
Equation~\eqref{eq_sign_rev_Burge}:
$$
|\BGenmat_{m}[n]|=
\sum_{k=0}^n\sum_{i=0}^k(-1)^i\binom{k}{i}\binom{m(k-i)}{n}.
$$

A simple adjustment of the same approach allows us to prove the
equations for $\betass$ and $\betasw$. By Equation~\eqref{eq_mat(S)}, it
suffices to count (binary) Burge matrices whose vector of row sums is
equal to $\Delta(S)$. Given a matrix $A=(a_{ij})$ in $\G_m[n]$, let
$\hat{A}=(|a_{ij}|)$ denote the matrix obtained from $A$ by replacing
each linear order with its size. Also, let $\ones$ denote the all ones
vector. Then $\hproj{A}$ is the row sum vector of $\hat{A}$: its $i$th
component equals the total size of the linear orders on row $i$.
Letting
\begin{gather*}
\G_m(S)[n]=\lbrace A\in \G_m[n]:\hproj{A}=\Delta(S)\rbrace
\shortintertext{we have}
|\G_m(S)[n]|=
\sum_{k=0}^n\sum_{i=0}^k\binom{k}{i}
\multiset{k-i}{s_1}\multiset{k-i}{s_2-s_1}\dots\multiset{k-i}{n-s_r}.
\end{gather*}
Indeed, the same argument used to count $\G_m[n]$ holds, except that
here we have to pick the right sizes of the nonempty linear orders in
each row;
that is, we pick a weak composition of $s_1$ among the $k-i$ nonempty
columns for the first row, a weak composition of $s_2-s_1$ for the second
row, and so on.
Now, the sign $\sgn$ and the sign-reversing involution~$\gamma$
defined previously on $\G_m[n]$ work analogously on $\G_m(S)[n]$,
and the fixed points are the matrices in $\G_m(S)[n]$
where every column contains at least one nonempty linear order.
Further, due to the equality $\hproj{A}=S$, each row contains at
least one nonempty linear order.
To summarize, the fixed points are the Burge matrices
whose row sums are given by $\Delta(S)$, and identity~\eqref{eq_sign_rev_betass}
now immediately follows from identity~\eqref{eq_mat(S)}:
\begin{align*}
\betass &=|\Mat(S)[n]|\\
&=\sum_{k=0}^n\sum_{i=0}^k
(-1)^i\binom{k}{i}\multiset{k-i}{s_1}\multiset{k-i}{s_2-s_1}\dots\multiset{k-i}{n-s_r}.
\intertext{Finally, the same reasoning applies to binary matrices, yielding}
\betasw &=|\BMat(S)[n]|\\
&=\sum_{k=0}^n\sum_{i=0}^k
(-1)^i\binom{k}{i}\binom{k-i}{s_1}\binom{k-i}{s_2-s_1}\dots\binom{k-i}{n-s_r}.
\end{align*}
This completes the proof of Proposition~\ref{formulas_fixascset}.

We end this section with a remark about Cayley permutations whose ascent
set \emph{equals} the set $S=\{s_1,s_2,\dots,s_k\}$. Using the principle of
inclusion-exclusion, the number
$\lambda_n(S) := |\lbrace w\in\Cay[n]: \Asc(w) = S\rbrace|$
can be expressed in terms of the formula for $\betasw$, but the result
would be a rather unpleasing triple sum.  There is a nice determinant
formula for the number
$\beta_n(S)=|\lbrace v\in\Sym[n]:\Asc(v)= S\rbrace|$ of permutations of
$[n]$ whose ascent set is equal to $S$, namely
$\beta_n(S)=n!\det\bigl[1/(s_{j}-s_{i-1})\bigr]$.  See Stanley~\cite[p.\
229]{St} for further details. We have not been able to find a
corresponding formula for $\lambda_n(S)$.

\section{Binary Burge matrices}\label{section_SRI_2}

In Section~\ref{section_burmat}, we proved the species identities
\begin{align}
  \Sym\times\Genmat_m &= \bigl(\Bal\times \Set(mX)\bigr)\circ\Lin^c;\label{eqA} \\
  \Sym\times\Mat &= \left(\Bal\times\Bal\right)\circ\Lin^c. \label{eqB}
\end{align}
Here, we aim to define two sign-reversing involutions on the structures
of these species whose fixed points are
$(\Sym\times\BGenmat_m)$-structures and $(\Sym\times\BMat)$-structures,
respectively. This way we shall obtain new counting formulas for
$\BGenmat_{m}$ and $\BMat$.

Consider Equation~\eqref{eqA} and let $(w,A)$ be an
$(\Sym\times\Genmat_m)$-structure on $\ell$. By Lemma~\ref{lemma_LxBur},
we can identify $(w,A)$ with the matrix $M=w\cdot A$ obtained
by letting $w$ act on $A$. Following Example~\ref{example_wA},
\[\text{if }w=784652391\text{ and }
A=\begin{bmatrix}
\cdot & 5 & 67\\
123 & \cdot & \cdot\\
\cdot & \cdot & \cdot\\
4 & \cdot & 89
\end{bmatrix}\!,
\,\text{then }
M=\begin{bmatrix}
\cdot & 5 & 23\\
784 & \cdot & \cdot\\
\cdot & \cdot & \cdot\\
6 & \cdot & 91
\end{bmatrix}.
\]
Alternatively, in terms of the species on the right-hand side of
Equation~\eqref{eqA}, $M$ is obtained by arranging atoms (i.e., connected
linear orders) in a matrix where the column index of each atom is determined
by the first component of the cartesian product $\Bal\times \Set(mX)$, and
the row index is determined by the second component.
In the matrix $M$ of the previous example, the atoms are
$\lbrace 78,4,6,5,23,9,1 \rbrace$. The same reasoning
applies to the species of Equation~\eqref{eqB}.
In a pair $(w,A)\in\Sym\times\Mat$ the rows of $A$ are, however, required to
be nonempty. Hence the row index is determined by a $\Bal$-structure,
and we arrive at $\left(\Bal\times\Bal\right)\circ\Lin^c$.

Now, let
$T=\Sym\times\Genmat_m=\LCirc=\bigl(\Bal\times \Set(mX)\bigr)\circ\Lin^c$
be the species of Equation~\eqref{eqA}. We shall identify a $T$-structure
on $\ell$ with a matrix in $\LCircStr[\ell]$, but we will keep track of
its atoms by separating the atoms in each entry with vertical bars.
The usual example matrix $M$ is encoded as
$$
\phantom{\tau()}M=\begin{bmatrix}
\cdot & 5 & 23\\
78|4 & \cdot & \cdot\\
\cdot & \cdot & \cdot\\
6 & \cdot & 9|1
\end{bmatrix}.
$$
Define the sign of $M\in T[n]$ as $\sgn(M)=(-1)^{n-k}$, where $k$
is the number of atoms of $M$. Furthermore, define a mapping
$\tau:T[n]\to T[n]$ by letting $\tau(M)$ be the matrix obtained from $M$
in the following manner. Read the entries of $M$ in some canonical
order; e.g., read them in the same order as they appear in the
concatenation $\prod M$. Relative to this order, find the first entry
$m_{ij}$ of $M$ such that $|m_{ij}|\geq 2$ and transpose the
first two letters of $m_{ij}$. If there is no such entry
$m_{ij}$, then let $\tau(M)=M$. The matrix $M$ illustrated previously
has seven atoms and sign $\sgn(M)=(-1)^{9-7}=1$. Further,
$$
\tau(M)=\begin{bmatrix}
\cdot & 5 & 23\\
8|7|4 & \cdot & \cdot\\
\cdot & \cdot & \cdot\\
6 & \cdot & 9|1
\end{bmatrix}
$$
is obtained from $M$ by the transposition $78\mapsto 8|7$. Since
$\tau(M)$ has one more atom than $M$, the sign is reversed:
$\sgn(\tau(M))=-\sgn(M)$.

It is clear that $\tau$ is an involution on $T[n]$.
In particular, $M$ is a fixed point of $\tau$ if and only if each entry
of $M$ is a linear order of size at most one. If $(w,A)$ is the pair in
$\left(\Sym\times\Genmat_m\right)[n]$ corresponding to $M$, this is the same as
saying that $A$ is binary. In other words, the set of fixed points of
$\tau$ corresponds to the pairs in $\left(\Sym\times\BGenmat_m\right)[n]$.
On the other hand, if $M$ is not a fixed point, then $\tau(M)$ has the opposite
sign of $M$. Indeed, suppose that $\tau$ acts on $M$ by transposing
the two letters $yz$. If $y<z$, then $y$ and $z$ belong to
the same atom in $M$, but to different atoms in $\tau(M)$, and thus
$\tau(M)$ has one more atom than $M$. Conversely, if $y>z$, then $\tau(M)$
has one less atom than $M$. Therefore, $\tau$ is a sign-reversing
involution on $T$ and
\begin{align*}
|\BGenmat_{m}[n]|
&=\frac{1}{n!}\sum_{M\in\fix(\tau)}\sgn(M)\\
&=\frac{1}{n!}\sum_{M\in T[n]}\sgn(M)
=\frac{1}{n!}\sum_{k=0}^n(-1)^{n-k}\stone{n}{k}\fubini(k)m^k.
\end{align*}
The same reasoning applies to matrices in $T$ whose
every row contains at least one nonempty linear order, that is,
to the species $\left(\Bal\times\Bal\right)\circ\Lin^c$.
The signed sum of such matrices of size~$n$ is
$\sum_{k=0}^n(-1)^{n-k}\stone{n}{k}\fubini(k)^2$
and an analogous equation for $|\BMat[n]|$ follows.
We collect the formulas derived in this section in the next proposition.

\begin{proposition}\label{prop_sri_binary}
For each $n,m\ge 0$,
\begin{align*}
|\BGenmat_{m}[n]|&=\frac{1}{n!}\sum_{k=0}^n (-1)^{n-k}\stone{n}{k}\fubini(k)m^k;\\
|\BMat[n]|&=\frac{1}{n!}\sum_{k=0}^n(-1)^{n-k}\stone{n}{k}\fubini(k)^2.
\end{align*}
\end{proposition}

Let $\Lin^c_{(-1)}$ denote the species of \emph{signed connected linear
orders}, i.e., connected linear orders $w\in \Lin^c[n]$ with sign
$(-1)^{n-1}$. In particular,
$$
\Lin^c_{(-1)}(x)=\sum_{n\ge 0}(-1)^{n-1} (n-1)!\frac{x^n}{n!}
=\log\left(1+x\right).
$$
The sign-reversing involutions that lead to
Proposition~\ref{prop_sri_binary} are embodied by the equations in the
following proposition (whose straightforward proof we omit).
\begin{proposition}\label{species_sri}
  For $m\geq 0$,
  \begin{align*}
    \Sym\times\Genmat_m^{01} &= \bigl(\Bal\times \Set(mX)\bigr)\circ\Lin^c_{(-1)}; \\
    \Sym\times\BMat &= \left(\Bal\times\Bal\right)\circ\Lin^c_{(-1)}.
  \end{align*}
\end{proposition}

We end this section by noting that the sign-reversing
involution~$\tau$ can be applied to the species $\Sym\times\Mat^{s,t}$
and that the result is a formula for the strict two-sided Caylerian polynomials
that is analogous to Theorem~\ref{thm_twocay} for the
weak counterpart:
\begin{equation}\label{eq_twocaystrict}
\twocayhat^{\circ}_n(s,t)=
\frac{1}{n!}\sum_{k=0}^n(-1)^{n-k}\stone{n}{k}\bigl|\Bal^s[k]\bigr|\bigl|\Bal^t[k]\bigr|.
\end{equation}
A formula for the strict Caylerian polynomials can be obtained
in a similar fashion, this time in analogy with
Theorem~\ref{thm_caylerian_explicit}:
\begin{equation}\label{eq_caylerianstr_explicit}
C^{\circ}_n(t)=
\frac{1}{n!}\sum_{k=0}^n (-1)^{n-k}\stone{n}{k}\fubini(k) \sum_{i=0}^k\sttwo{k}{i}i!(t-1)^{n-i}.
\end{equation}

\section{Final remarks}\label{sec_final}

In our recent paper~\cite{CC2}, we developed a
framework relating the Caylerian polynomials and their variants
with Burge matrices and Burge words.
Here, using combinatorial species and sign-reversing involutions,
we built on the interplay between these structures
to obtain several counting formulas and species equations, most of which are collected
in tables~\ref{table_burge}, \ref{table_caylerian} and~\ref{table_species}.
For a list of open problems and suggestions for future work,
we refer to the same paper~\cite{CC2}. We end with a couple of further
remarks and questions.

The OEIS~\cite{Sl} entries for the counting sequences of Burge
matrices (A120733) and binary Burge matrices (A101370) contain
the generating functions
\begin{align*}
\sum_{n\ge 0}|\Mat[n]|x^n &= \sum_{m\ge 0}\frac{1}{2^{m+1}}(\Sym\times\Genmat_m)(x);\\
\sum_{n\ge 0}|\BMat[n]|x^n &= \sum_{m\ge 0}\frac{1}{2^{m+1}}(\Sym\times\BGenmat_m)(x).
\end{align*}
Can these two identities be proved with the tools developed here?

Maia and Mendez~\cite[equation~(92)]{MM} defined the \emph{modified
arithmetic product} of two species $F$ and $G$ as
$$
\left(F\ModProd G\right)[U] = \sum_{(\pi,\tau)} F[\pi]\times G[\tau],
$$
where the sum ranges over all the partial rectangles $(\pi,\tau)$
over $U$. Further, they~\cite[equation~(116)]{MM} showed that
$$
\left(F\ModProd G\right)(x)=
\bigl(F(\Set_+)\times G(\Set_+)\bigr)(x)\circ\log(1+x).
$$
The notion of arithmetic product allows us to write
$\Sym\times\BMat=\Lin\ModProd \Lin$,
and thus
\[
\left(\Sym\times\BMat\right)(x)
=\left(\Lin\ModProd \Lin\right)(x)
=\left(\Bal\times \Bal\right)(x)\circ\log(1+x),
\]
which gives an alternative path to the second equation of Proposition~\ref{prop_sri_binary}.
Using another identity of Maia and Mendez~\cite[Equation (143)]{MM},
we also get
$$
|\BMat[n]| = \sum_{r,s\ge 0}\frac{1}{2^{r+s+2}}\binom{rs}{n}.
$$
We could not find a proof of the corresponding equation for $\Mat[n]$
(A120733~\cite{Sl}):
$$
|\Mat[n]| = \sum_{r,s\ge 0}\frac{1}{2^{r+s+2}}\multiset{rs}{n}.
$$

Munarini, Poneti and Rinaldi~\cite{MPR} considered \emph{matrix
compositions without zero rows}. In our setting, such matrix
compositions of~$n$ correspond to Burge matrices in $\Mat[n]$.
The interested reader is invited to consult Section~7 of their
paper in which they present more enumerative results concerning these matrices.

\begin{table}
\begin{center}
\renewcommand{\arraystretch}{3.5}
\begin{tabular}{ll}
\textbf{Burge matrices} & \textbf{Reference}\\
\midrule
$\displaystyle{|\Genmat_{m}[n]|=\sum_{\alpha\models n}\prod_{a\in\alpha}\multiset{m}{a}}$
& Equation~\eqref{eq_genmat}\\
$\displaystyle{|\Genmat_{m}[n]|=\frac{1}{n!}\sum_{k=0}^n\stone{n}{k}m^k \fubini(k)}$
& Proposition~\ref{prop_bal_mLc}\\
$\displaystyle{|\Genmat_{m}[n]|=\sum_{k,i}(-1)^i\binom{k}{i}\multiset{m(k-i)}{n}}$
& Equation~\eqref{eq_sign_rev_Burge}\\
$\displaystyle{|\BGenmat_{m}[n]|=\sum_{\alpha\models n}\prod_{a\in\alpha}\binom{m}{a}}$
& Proposition~\ref{prop_enum_bgenm}\\
$\displaystyle{|\BGenmat_{m}[n]|=\frac{1}{n!}\sum_{k=0}^n (-1)^{n-k}\stone{n}{k}\fubini(k)m^k}\qquad$
& Proposition~\ref{prop_sri_binary}\\
$\displaystyle{|\BGenmat_{m}[n]|=
\sum_{k,i}(-1)^i\binom{k}{i}\binom{m(k-i)}{n}}$
& Equation~\eqref{eq_sign_rev_BinBurge}\\
$\displaystyle{|\Mat[n]|=\frac{1}{n!}\sum_{k=0}^n\fubini(k)^2\stone{n}{k}}$
& Proposition~\ref{prop_spec_burmat}\\
$\displaystyle{|\BMat[n]|=\frac{1}{n!}\sum_{k=0}^n(-1)^{n-k}\stone{n}{k}\fubini(k)^2}$
& Proposition~\ref{prop_sri_binary}\\
$\displaystyle{\sum_{n\ge 0}|\Genmat_{m}[n]|x^n=\frac{(1-x)^m}{2(1-x)^m-1}}$
& Proposition~\ref{prop:ogf_genmat}\\
$\displaystyle{\sum_{n\ge 0}|\BGenmat_{m}[n]|x^n=\frac{1}{2-(1+x)^m}}$
& Proposition~\ref{prop:ogf_genmat}
\end{tabular}
\end{center}
\caption{Formulas for (variants of) Burge matrices}\label{table_burge}
\end{table}

\begin{table}
\begin{center}
\renewcommand{\arraystretch}{3.5}
\begin{tabular}{ll}
\textbf{Caylerian polynomials} & \textbf{Reference}\\
\midrule
$\displaystyle{C_n(t)=
\frac{1}{n!}\sum_{k,i}
\stone{n}{k}\fubini(k)\sttwo{k}{i}i!(t-1)^{n-i}}$
& Theorem~\ref{thm_caylerian_explicit}\\
$\displaystyle{C^{\circ}_n(t)=
\frac{1}{n!}\sum_{k,i}
(-1)^{n-k}\stone{n}{k}\fubini(k)\sttwo{k}{i}i!(t-1)^{n-i}}$
& Equation~\eqref{eq_caylerianstr_explicit}\\
$\displaystyle{\sum_{n\ge 0}\frac{tC_n(t)}{(1-t)^{n+1}}x^n=
\sum_{m\ge 1}\frac{(1-x)^{m}}{2(1-x)^{m}-1}t^m}$
& Theorem~\ref{theorem_genfun_n}\\
$\displaystyle{\sum_{n\ge 0}\frac{tC^{\circ}_n(t)}{(1-t)^{n+1}}x^n=
\sum_{m\ge 1}\frac{1}{2-(1+x)^{m}}t^m}$
& Theorem~\ref{theorem_genfun_n}\\
$\displaystyle{\twocayhat_n(s,t)=
\frac{1}{n!}\sum_{k=0}^n\stone{n}{k}\bigl|\Bal^s[k]\bigr|\bigl|\Bal^t[k]\bigr|}$
& Theorem~\ref{thm_twocay}\\
$\displaystyle{\twocayhat^{\circ}_n(s,t)=
\frac{1}{n!}\sum_{k=0}^n(-1)^{n-k}\stone{n}{k}\bigl|\Bal^s[k]\bigr|\bigl|\Bal^t[k]\bigr|}$
& Equation~\eqref{eq_twocaystrict}\\
$\displaystyle{\betasw=
\sum_{k,i}(-1)^i\binom{k}{i}\prod_{j=0}^r\binom{k-i}{s_{j+1}-s_{j}}}$
& Equation~\eqref{eq_sign_rev_betasw}\\
$\displaystyle{\betass=
\sum_{k,i}(-1)^i\binom{k}{i}\prod_{j=0}^r\multiset{k-i}{s_{j+1}-s_{j}}}$
& Equation~\eqref{eq_sign_rev_betass}
\end{tabular}
\end{center}
\caption{Formulas for Caylerian polynomials}\label{table_caylerian}
\end{table}

\begin{table}
\begin{center}
\renewcommand{\arraystretch}{2}
\begin{tabular}{ll}
\textbf{Species equation} & \textbf{Reference}\\
\midrule
$\Sym\times\Genmat_m = \Lin\circ(L^m-1)$
& Lemma~\ref{lemma_LxBur}\\
$\Sym\times\BGenmat_m=\Lin\circ\bigl((1+X)^m-1\bigr)$
& Proposition~\ref{prop_enum_bgenm}\\
$\Sym\times\Genmat_m=\bigl(\Bal\times \Set(mX)\bigr)\circ\Lin^c$
& Proposition~\ref{prop_bal_mLc}\\
$\Sym\times\BGenmat_m=\bigl(\Bal\times \Set(mX)\bigr)\circ\Lin^c_{(-1)}$
& Proposition~\ref{species_sri}\\
$\Sym\times\Mat=
\left(\Bal\times\Bal\right)\circ\Lin^c$
& Proposition~\ref{prop_spec_burmat}\\
$\Sym\times\BMat=\left(\Bal\times\Bal\right)\circ\Lin^c_{(-1)}$
& Proposition~\ref{species_sri}\\
$\Sym\times\Mat^{s,t}=\left(\Bal^s\times\Bal^t\right)\circ\Lin^c$
& Equation~\eqref{eq_symXmatst}
\end{tabular}
\end{center}
\caption{Species equations}\label{table_species}
\end{table}

\FloatBarrier

\end{document}